\documentclass[11pt,reqno]{amsart}
\usepackage{indentfirst, amssymb,amsmath,amsthm, mathrsfs,cite,graphicx,float}
\newtheorem*{theoA}{Theorem A}
\newtheorem*{theoB}{Theorem B}
\newtheorem*{theoC}{Theorem C}

\newtheorem{theo}{Theorem}[section]
\newtheorem{lem}{Lemma}[section]

\newtheorem{rem}{Remark}[section]

\newtheorem{que}{Question}[section]

\newtheorem{open problem}{Open problem}[section]
\newcommand{\pa}{\partial}
\newcommand{\ol}{\overline}
\newcommand{\be}{\begin{equation}}
\newcommand{\ee}{\end{equation}}
\newcommand{\bs}{\begin{small}}
\newcommand{\es}{\end{small}}
\newcommand{\beas}{\begin{eqnarray*}}
\newcommand{\eeas}{\end{eqnarray*}}
\newcommand{\bea}{\begin{eqnarray}}
\newcommand{\eea}{\end{eqnarray}}
\renewcommand{\epsilon}{\varepsilon}
\numberwithin{equation}{section}

\begin{document}
\title[Improved and refined inequalities for harmonic mappings ]{A unified framework for Bohr-type inequalities using multiple Schwarz functions}
\author[R. Biswas and R. Mandal]{Raju Biswas and Rajib Mandal}
\date{}
\address{Raju Biswas, Department of Mathematics, Raiganj University, Raiganj, West Bengal-733134, India.}
\email{rajubiswasjanu02@gmail.com}
\address{Rajib Mandal, Department of Mathematics, Raiganj University, Raiganj, West Bengal-733134, India.}
\email{rajibmathresearch@gmail.com}

\let\thefootnote\relax
\begin{abstract} This paper introduces a unified framework for Bohr-type inequalities by incorporating multiple Schwarz functions into the majorant series for $K$-quasiconformal harmonic mappings in the unit disk $\mathbb{D} := \{z\in\mathbb{C} : |z| < 1\}$. In this study, we establish several improved and refined versions of the Bohr inequality that generalize and interconnect numerous known results. Our approach not only systematically recovers the existing theorems as special cases but also generates new results that are inaccessible through single-function methods.
\end{abstract}
\maketitle
\section{Introduction and preliminaries}
\noindent Let $H_\infty$ denote the collection of analytic functions in the unit disk $\mathbb{D}:=\{z\in\mathbb{C}: |z|<1\}$, characterized by the supremum norm $\Vert f\Vert_\infty:=\sup_{z\in\mathbb{D}}|f(z)|<\infty$.
For a function $f\in H_\infty$ represented by the Taylor series expansion $f(z)=\sum_{n=0}^{\infty} a_nz^n$, then
\bea\label{e2} \sum_{n=0}^{\infty}|a_n|r^n\leq \Vert f\Vert_\infty\quad\text{for}\quad|z|= r\leq\frac{1}{3}.\eea
In this context, the quantity $1/3$ is referred to as the Bohr radius, which cannot be improved. Inequality (\ref{e2}) is known as the classical Bohr inequality.
Indeed, Bohr \cite{B1914} proved that inequality (\ref{e2}) holds for $r\leq 1/6$. Subsequently, Schur, Wiener, and Riesz independently proved that the inequality (\ref{e2}) holds for $r\leq 1/3$. The proofs of Bohr and Wiener can be found in \cite{B1914}.\\[2mm]
\indent In addition to the Bohr radius, the Rogosinski radius is another important concept (see \cite{R1923}). 
Building upon the concept of the Rogosinski radius, Kayumov {\it et al.}\cite{KKP2021} introduced the Bohr-Rogosinski radius,  defined as the largest value $r_0\in(0, 1)$ such that the inequality $R_N^f(z)\leq 1$ holds for $|z|< r_0$, where
\beas R_N^f(z):=|f(z)|+\sum_{n=N}^{\infty}|a_n||z|^n, \quad N\in\mathbb{N}\eeas
for the analytic function $f$ in $\mathbb{D}$ with $|f(z)|\leq 1$ in $\Bbb{D}$. Kayumov {\it et al.}\cite{KKP2021} have established several results concerning the Bohr-Rogosinski radius.\\[2mm]
\indent Before proceeding with the discussion, it is necessary to establish certain notations.
Let $\mathcal{B}=\{f\in H_\infty :\Vert f\Vert_\infty\leq 1\}$ and 
\beas\mathcal{B}_m=\left\{f\in\mathcal{B}: f(0)=f'(0)=\cdots=f^{(m-1)}(0)=0\quad\text{and}\quad f^{(m)}(0)\not=0\right\},\quad m\in\mathbb{N}.\eeas
\noindent Let $\mathbb{D}_ r:=\{z\in\mathbb{C}: |z|< r \}$ for $0< r<1$. For an analytic function $h$ in $\mathbb{D}$, let $S_ r(h)$ denote the planar integral  
\beas S_ r(h)=\int_{\mathbb{D}_ r} |h'(z)|^2 dA(z).\eeas
If $h(z)=\sum_{n=0}^\infty a_n z^n$, then it is well-known that $S_ r(h)/\pi=\sum_{n=1}^\infty n|a_n|^2 r^{2n}$. Moreover, when $h$ is univalent in the unit disk $\mathbb{D}$, the expression $S_ r(h)$ corresponds to the area of the image $h(\mathbb{D}_ r)$. In addition, if $f(z)$ is analytic in $\mathbb{D}$, then we say that $f$ is quasi-subordinate to $h$, denoted by $f(z)\prec_q h(z)$ in $\mathbb{D}$, if there exist analytic functions $\phi$ and $\omega$ in $\mathbb{D}$ satisfying $\omega(0)=0$, $|\omega(z)|< 1$, and $|\phi(z)|\leq 1$ for $|z|<1$ such that $f(z)=\phi(z) h(\omega(z))$ for $z\in\Bbb{D}$ (see \cite{R1973}). \\[2mm]
In 2018, Kayumov and Ponnusamy \cite{KP2018} obtained several improved versions of Bohr's inequality for bounded analytic functions in the unit disk $\mathbb{D}$. Here, we recall one such version.
\begin{theoA}\cite{KP2018}
Let $f(z)=\sum_{n=0}^{\infty} a_nz^n$ be analytic in $\mathbb{D}$ such that $|f(z)|\leq 1$ in $\mathbb{D}$ and $S_ r$ denote the area of the image of the subdisk $|z|< r$ under mapping $f$. Then,
\beas \sum_{n=0}^{\infty} |a_n| r^n+\frac{16}{9}\left(\frac{S_ r(f)}{\pi}\right)\leq 1\quad\text{for}\quad r\leq\frac{1}{3}.\eeas
The numbers $1/3$ and $16/9$ cannot be improved further. \end{theoA}
\indent Following the foundational work of Kayumov {\it et al.} \cite{KKP2021, KP2018}, a significant number of authors have investigated various modifications to Bohr-type inequalities (see \cite{AKP2020,LSX2018,KKP2021}).\\[1mm]
Consider the function $f = u + iv$, which is complex-valued and defined in terms of $z = x + iy$ within a simply connected domain $\Omega$. If the function $f$ is twice 
continuously differentiable and satisfies the Laplace equation $\Delta f = 4f_{z\ol z} = 0$ in $\Omega$, then $f$ is harmonic in $\Omega$. The derivatives are defined as $f_z = (f_x - i f_y)/2$ and $f_{\ol{z}} = (f_x + i f_y)/2$. It is important to note that every 
harmonic mapping $f$ can be expressed in its canonical form as $f = h + \ol g$, where $h$ and $g$ are analytic in $\Omega$ and are referred to as the analytic and co-analytic parts of 
$f$, respectively. This representation is unique up to an additive constant (see \cite{D2004}). The inverse function theorem, along with Lewy's result \cite{L1936}, establishes that a 
harmonic function $f$ is locally univalent in $\Omega$ if, and only if, the Jacobian of $f$, given by $J_f(z):=|h'(z)|^2-|g'(z)|^2$, is nonzero in $\Omega$. A harmonic mapping $f$ is 
locally univalent and sense-preserving in $\Omega$ if, and only if, $J_f (z)>0$ in $\Omega$, or equivalently, if $h' \neq 0$ in $\Omega$ and the dilatation $\omega = g'/h'$ of $f$ satisfies the condition $|\omega|\leq 1$ in $\Omega$ (see \cite{L1936}).\\[2mm]
\indent If a locally univalent and sense-preserving harmonic mapping $f = h+\ol{g}$ satisfies the condition $|g'(z)/h'(z)| \leq k < 1$ for $z\in\mathbb{D}$, then $f$ is called a $K$-quasiconformal harmonic mapping on $\mathbb{D}$, where $K = (1+k)/(1-k) \geq 1$ (see \cite{K2008,M1968}). It is evident that as $k\to 1$, the limiting case corresponds to $K \to\infty$.\\[2mm]
\indent In 2018, Kayumov {\it et al.} \cite{KPS2018} advanced the classical Bohr theorem by introducing a harmonic extension, which led to the following insights.
\begin{theoB}\cite{KPS2018} Consider the function $f(z)=h(z)+\ol{g(z)}=\sum_{n=0}^\infty a_n z^n+\ol{\sum_{n=1}^\infty b_n z^n}$, which is a sense-preserving and $K$-quasiconformal harmonic mapping in $\mathbb{D}$, where $h(z)$ is bounded in $\mathbb{D}$. Then,
\beas\sum_{n=0}^\infty |a_n| r^n+\sum_{n=1}^\infty |b_n| r^n\leq \Vert h(z)\Vert_\infty\quad\text{for}\quad r\leq \frac{K+1}{5K+1}.\eeas
The value $(K+1)/(5K+1)$ is optimal and cannot be improved. Furthermore,
\beas |a_0|^2+\sum_{n=1}^\infty (|a_n|+|b_n|) r^n\leq \Vert h(z)\Vert_\infty\quad\text{for}\quad r\leq \frac{K+1}{3K+1}.\eeas
The number $(K+1)/(3K+1)$ cannot be improved.\end{theoB}
\begin{theoC}\cite{KPS2018} Consider the function $f(z)=h(z)+\ol{g(z)}=\sum_{n=0}^\infty a_n z^n+\ol{\sum_{n=2}^\infty b_n z^n}$, which represents a sense-preserving and $K$-quasiconformal harmonic mapping in $\mathbb{D}$, where $h(z)$ is bounded in $\mathbb{D}$. Then,
\beas\sum_{n=0}^\infty |a_n| r^n+\sum_{n=2}^\infty |b_n| r^n\leq \Vert h(z)\Vert_\infty\quad\text{for}\quad r\leq r_K,\eeas
where $ r_K$ is the positive root of the equation
\beas \frac{ r}{1- r}+\left(\frac{K-1}{K+1}\right) r^2\sqrt{\frac{1+ r^2}{(1- r^2)^3}}\sqrt{\frac{\pi^2}{6}-1}=\frac{1}{2}.\eeas
The number $ r_K$ cannot be replaced by any number greater than $R=R(K)$, where $R$ is the positive root of the equation
\beas \frac{4R}{1-R}\left(\frac{K}{K+1}\right)+2\left(\frac{K-1}{K+1}\right)\log(1-R)=1.\eeas\end{theoC}
Additionally, several authors have studied other aspects of Bohr's inequality.
We refer to \cite{ABM2024,AKP2019,AAH2022,AH2021,1AH2021,AH2022,BDK2004,BB2004,EPR2019,EPR2021,FR2010,HLP2020,IKP2020,1KP2018,LP2019,MBG2024,PW2020} and the references listed therein for an in-depth investigation on the Bohr radius.\\[2mm]
\noindent The aforementioned results show a development from classical settings to complex function classes. In this study, we demonstrate that the simultaneous incorporation of 
multiple Schwarz functions creates a powerful, comprehensive framework. This approach not only allows us to recover the aforementioned theorems as special cases but also yields 
new, sharper results that are inaccessible through single-function methods. This brings us to the following central questions of this paper:
\begin{que}\label{Q1} Can the theory of Bohr-type inequalities for harmonic mappings be systematically extended through the simultaneous incorporation of multiple Schwarz functions, namely $\omega_p(z)$, $\omega_m(z)$ and $\omega_q(z)$, for the analytic and co-analytic parts?
\end{que}
\begin{que}\label{Q2} Is it possible to not only generalize but also sharpen \textrm{Theorem C} by obtaining a single equation for the optimal Bohr radius that improves upon the original two-sided bound?
\end{que}
\noindent The purpose of this paper is primarily to provide affirmative answers to Questions \ref{Q1} and \ref{Q2}.
\section{Key lemmas}
\noindent The following lemmas are required to prove our main results.
\begin{lem}\label{lem1} \cite[Pick's invariant form of Schwarz lemma]{K2006} If $f$ is analytic in $\mathbb{D}$ with $|f(z)|\leq1$, then 
\beas |f(z)|\leq \frac{|f(0)|+|z|}{1+|f(0)||z|}\quad\text{for}\quad z\in\mathbb{D}.\eeas\end{lem}
\begin{lem}\cite{DP2008}\label{lem2} For an analytic function $f$ in $\mathbb{D}$ with $|f(z)|\leq1$,  we have 
\beas \frac{\left|f^{(n)}(z)\right|}{n!}\leq \frac{1-|f(z)|^2}{(1-|z|)^{n-1}(1-|z|^2)}\quad{and}\quad |a_n|\leq 1-|a_0|^2\;\;\text{for}\quad n\geq 1\quad \text{and}\quad |z|<1.\eeas\end{lem}
\begin{lem}\cite{KPS2018}\label{lem3} Let $h(z)=\sum_{n=0}^\infty a_nz^n$ and $g(z)=\sum_{n=0}^\infty b_nz^n$ be two analytic functions in $\mathbb{D}$ such that $|g’(z)|\leq k|h’(z)|$ in $\mathbb{D}$ for some $k\in [0,1)$. If $|h(z)|\leq 1$ in $\mathbb{D}$, then
\beas\sum_{n=1}^\infty |b_n|^2 r^n\leq k^2 \sum_{n=1}^\infty |a_n|^2 r^n\;\;\text{for}\;\;|z|= r<1.\eeas
\end{lem}
\noindent Using the concept of quasi-subordination and the result derived in \cite{AKP2019}, Liu {\it et al.} \cite{LPW2020} established the following result.
\begin{lem}\cite[Proof of Theorem 2]{LPW2020}\label{lem30} Let $h(z)=\sum_{n=0}^\infty a_nz^n$ and $g(z)=\sum_{n=0}^\infty b_nz^n$ be two analytic functions in $\mathbb{D}$ such that $|g’(z)|\leq k|zh’(z)|$ in $\mathbb{D}$ for $k\in [0,1)$. Then,
\beas\sum_{n=1}^\infty n |b_n| r^{n-1}\leq k \sum_{n=1}^\infty n|a_n| r^n\quad\text{for}\quad |z|= r\leq 1/3.\eeas
\end{lem}
\begin{lem}\cite{LLP2021}\label{lem5} Let $f$ be analytic in $\mathbb{D}$ such that $|f(z)|\leq1$ in $\mathbb{D}$. For any $N\in\mathbb{N}$, the following inequality holds:
\beas\sum_{n=N}^\infty |a_n| r^n+\text{sgn}(t)\sum_{n=1}^t |a_n|^2 \frac{ r^N}{1- r}+\left(\frac{1}{1+a_0}+\frac{ r}{1- r}\right)\sum_{n=t+1}^\infty |a_n|^2  r^{2n}\leq \frac{(1-|a_0|^2) r^N}{1- r}\eeas
for $r\in[0,1)$, where $t=\lfloor (N-1)/2\rfloor$ and $\lfloor x\rfloor$ denotes the largest integer not exceeding the real number $x$.
\end{lem}
\section{Main results}
\noindent In the subsequent result, we obtain a sharp version of \textrm{Theorem C} in the context of harmonic mappings, incorporating multiple Schwarz functions.
\begin{theo}\label{T6} Let $f(z)=h(z)+\ol{g(z)}=\sum_{n=0}^\infty a_n z^n+\ol{\sum_{n=2}^\infty b_n z^n}$ be a sense-preserving, $K$-quasiconformal harmonic mapping in $\mathbb{D}$, and $\omega_m\in\mathcal{B}_m$ for $m\geq 1$. If $h(z)$ is bounded in $\mathbb{D}$, then
\beas\sum_{n=0}^\infty |a_n| |\omega_p(z)|^n+\sum_{n=2}^\infty |b_n| |\omega_m(z)|^n \leq \Vert h(z)\Vert_\infty\quad\text{for $ r\leq \min\left\{r_{p,m,k},\;1/\sqrt[m]{3}\right\}$,}\eeas
where 
$ r_{p,m,k}\in(0,1)$ is the smallest positive root of the equation 
\beas \frac{2r^{p}}{1- r^p}+ 2k\left(\frac{ r^{m}}{1- r^m}+\log(1- r^m)\right)-1=0,\eeas
where $k=(K-1)/(K+1)$. If $r_{p,m,k}\leq 1/\sqrt[m]{3}$, then the number $r_{p,m,k}$ cannot be improved.
\end{theo}
\begin{proof}
For simplicity, let $\Vert h(z)\Vert_\infty\leq 1$. In view of \textrm{Lemma \ref{lem2}}, we have $|a_n|\leq 1-|a_0|^2$ for $n\geq 1$.
Since $f$ is locally univalent and $K$-quasiconformal sense-preserving harmonic mapping in $\mathbb{D}$ with $g'(0)=b_1=0$, the classical Schwarz Lemma gives that the dilatation 
$\omega=g'/h'$ is analytic in $\mathbb{D}$ and $|\omega(z)|\leq k|z|$, {\it i.e.}, $|g'(z)|\leq k|zh'(z)|$ in $\mathbb{D}$, where $k =(K-1)/(K+1)\in[0,1)$.  
In view of \textrm{Lemma \ref{lem30}}, we have
\bea\label{p2} \sum_{n=2}^\infty n |b_n| r^{n-1}\leq k \sum_{n=1}^\infty n|a_n| r^n\leq k(1-a^2)\sum_{n=1}^\infty n r^n=k(1-a^2)\frac{ r}{(1- r)^2}\eea
for $|z|= r\leq 1/3$, where $|a_0|=a\in[0,1)$. Integrate (\ref{p2}) on $[0,  r]$, we have 
\be\label{p6}\sum_{n=2}^\infty |b_n| r^{n}\leq k(1-a^2)\int_{0}^ r\frac{x}{(1-x)^2} dx=k(1-a^2)\left(\frac{ r}{1- r}+\log(1- r)\right)\;\;\text{for}\;\; r\leq \frac{1}{3}.\ee
It is evident that $H_1(r)=r/(1-r)+\log(1-r)$ is a monotonically increasing functions of $r$ and it follows that $H_1(r)\geq H_1(0)=0$. As $\omega_m\in\mathcal{B}_m$, in view of the classical Schwarz Lemma, we have $|\omega_m(z)|\leq |z|^m$.
Using (\ref{p6}), we have 
\beas \sum_{n=2}^\infty |b_n| |\omega_m(z)|^n\leq \sum_{n=2}^\infty |b_n| r^{mn}\leq k(1-a^2)\left(\frac{ r^m}{1- r^m}+\log(1- r^m)\right)\;\;\text{for}\;\; r\leq \frac{1}{\sqrt[m]{3}}.\eeas
Similarly, we obtain
\beas \sum_{n=1}^\infty |a_n| |\omega_p(z)|^n\leq \sum_{n=1}^\infty |a_n| r^{pn} \leq (1-a^2)\frac{ r^{p}}{1- r^p}.\eeas
Therefore, we have
\beas&&\sum_{n=0}^\infty |a_n| |\omega_p(z)|^n+\sum_{n=2}^\infty |b_n| |\omega_m(z)|^n\nonumber\\
&\leq&a+(1-a^2)\frac{ r^{p}}{1- r^p}+ k(1-a^2)\left(\frac{ r^{m}}{1- r^m}+\log(1- r^m)\right)\nonumber\\[2mm]
&=&1+(1-a) H_2(a, r),\eeas
where 
\beas H_2(a, r)=\frac{(1+a)r^{p}}{1- r^p}+ k(1+a)\left(\frac{ r^{m}}{1- r^m}+\log(1- r^m)\right)-1\eeas
and the first inequality holds for any $r\leq 1/\sqrt[m]{3}$.
Differentiating $H_2(a, r)$ partially with respect to $a$, we obtain 
\beas\frac{\pa}{\pa a}H_2(a, r)&=&\frac{r^{p}}{1- r^p}+ k\left(\frac{ r^{m}}{1- r^m}+\log(1- r^m)\right)\geq 0.\eeas
This shows that $H_2(a, r)$ is a monotonically increasing function of $a\in[0,1)$ and
\beas H_2(a, r)\leq \lim_{r\to1^-}H_2(a, r)=\frac{2r^{p}}{1- r^p}+ 2k\left(\frac{ r^{m}}{1- r^m}+\log(1- r^m)\right)-1\leq 0\eeas 
for $r\leq  r_{p,m,k}$, where $ r_{p,m,k}\in(0,1)$ is the smallest positive root of the equation 
\bea\label{p3} G_{p,m,k}(r):=\frac{2r^{p}}{1- r^p}+ 2k\left(\frac{ r^{m}}{1- r^m}+\log(1- r^m)\right)-1=0,\eea
where $k=(K-1)/(K+1)$. Therefore,
\beas\sum_{n=0}^\infty |a_n| |\omega_p(z)|^n+\sum_{n=2}^\infty |b_n| |\omega_m(z)|^n\leq 1\quad\text{for}~r\leq \min\left\{1/\sqrt[m]{3},\; r_{p,m,k}\right\}.\eeas
\indent We show that if $r_{p,m,k}\leq 1/\sqrt[m]{3}$, then the number $r_{p,m,k}$ cannot be improved. 
Consider the function $f_0(z)=h_0(z)+\ol{g_0(z)}$ in $\mathbb{D}$ such that 
\beas h_0(z)=\frac{a+z}{1+az}=A_0+\sum_{n=1}^\infty A_n z^n\quad\text{and}\quad g_0'(z)=\mu k zh_4'(z)\eeas
with $\omega_m(z)=z^m$ for $m\geq 1$, where $A_0=a$, $A_n=(1-a^2)(-a)^{n-1}$ for $n\geq 1$, $a\in[0,1)$, $|\mu|=1$ and $k=(K-1)/(K+1)$. If $g_0(z)=\sum_{n=2}^\infty B_n z^n$, then
\beas B_n=k\mu \left(\frac{n-1}{n}\right)(1-a^2)(-a)^{n-2}\quad\text{ for}\quad n\geq 2.\eeas
Therefore, we have
\beas&& \sum_{n=0}^\infty |A_n| r^{pn}+\sum_{n=2}^\infty |B_n| r^{mn}\\[2mm]
&=&a+(1-a^2) r^p\sum_{n=1}^\infty (a r^p)^{n-1}+(1-a^2)k r^{2m}\sum_{n=2}^\infty \frac{n-1}{n}(a r^m)^{n-2}\\[2mm]
&=&a+\frac{(1-a^2) r^p}{1-a r^p}+(1-a^2)k\frac{a r^m+(1-a r^m )\log \left(1-a r^m\right)}{a^2 \left(1-a r^m\right)}\\[2mm]
&=&1+(1-a)H_3(a, r),\eeas
where 
\beas H_3(a, r)=\frac{(1+a) r^p}{1-a r^p}+(1+a)k\frac{a r^m+(1-a r^m )\log \left(1-a r^m\right)}{a^2 \left(1-a r^m\right)}-1.\eeas
It is evident that
\beas \lim_{a\to1^-} H_3(a, r)=\frac{2r^{p}}{1- r^p}+ 2k\left(\frac{ r^{m}}{1- r^m}+\log(1- r^m)\right)-1>0\eeas
for $ r> r_{p,m,k}$, where $k=(K-1)/(K+1)$. This shows that the number $r_{p,m,k}$ cannot be improved.
This completes the proof.
\end{proof}
\begin{rem} Whenever $m\geq p$, then $r_{p,m,k}$ is the unique positive root of equation (\ref{p3}), and $r_{p,m,k}\leq 1/\sqrt[m]{3}$.
It is evident that
\beas G_{p,m,k}'(r)=\frac{2 \left(k m\; r^{2 m-1} \left(1-r^p\right)^2+p \left(1-r^m\right)^2 r^{p-1}\right)}{\left(1-r^m\right)^2 \left(1-r^p\right)^2}\geq0\quad\text{for $ r\in[ 0,1)$, }\eeas
which shows that $G_{p,m,k}( r)$ is a monotonically increasing function of $r$ with $G_{p,m,k}(0)=-1$ and 
\beas G_{p,m,k}(1/\sqrt[m]{3})&=&\left(\frac{2\left(\frac{1}{3}\right)^{p/m}}{1- \left(\frac{1}{3}\right)^{p/m}}-1\right)+2k(1/2+\log(2/3))\\[2mm]
&=&\frac{3\left(\left(\frac{1}{3}\right)^{p/m}-\frac{1}{3}\right)}{1- \left(\frac{1}{3}\right)^{p/m}}+2k(1/2+\log(2/3)).\eeas
Note that $(1/3)^{p/m}\geq (1/3)$ implies $(1/3)^{p}\geq (1/3)^m$ and the function $a^x=\exp(x\log(a))$ is a strictly decreasing function of $x$ for $0<a<1$, thus, we have $p\leq m$. Therefore, $G_{p,m,k}(1/\sqrt[m]{3})>0$ for $m\geq p$.
Therefore, $r_{p,m,k}\in\left(0,1/\sqrt[m]{3}\right)$ is the unique root of the equation (\ref{p3}), whenever $m\geq p$.
\end{rem}
\begin{rem}\label{rem2}
We discuss some special cases of \textrm{Theorem \ref{T6}} as well as several useful observations.
\begin{enumerate}
\item[(i)] By setting $p=m=1$ and $\omega_1(z)=z$ in \textrm{Theorem \ref{T6}}, we obtain the sharp version of \textrm{Theorem C}, {\it i.e.,}
\beas\sum_{n=0}^\infty |a_n| r^n+\sum_{n=2}^\infty |b_n| r^n \leq \Vert h(z)\Vert_\infty\quad\text{for $ r\leq  r_0$,}\eeas
where 
$r_0\in(0,1/3)$ is the unique positive root of the equation 
\beas \frac{2r}{1- r}+ \frac{2(K-1)}{(K+1)}\left(\frac{ r}{1- r}+\log(1- r)\right)-1=0.\eeas
The number $r_0$ is sharp.  This provides a conclusive answer to Question \ref{Q2}. The radius $r_0$ is now characterized by a single equation, which simplifies the two-sided estimate found in the original \textrm{Theorem C}.
\item[(ii)] By setting $p=1$ and allowing $m\to\infty$ in \textrm{Theorem \ref{T6}}, we derive the classical Bohr inequality. This demonstrates that our result on harmonic mapping inherently includes the classical Bohr inequality as a limiting case, thereby connecting the two significant research domains.
\end{enumerate}
\end{rem}
\noindent As shown in Table \ref{tab0}, the upper bound $1 /\sqrt[m] {3}$ increases as $m$ becomes larger. This sets the stage for understanding the specific radii $r_{p,m,k}$ calculated in Figure \ref{fig0}.
\begin{table}[H]
\centering
\caption{The values of $1/\sqrt[m]{3}$ for different values of $m\in\mathbb{N}$}
\label{tab0}
\begin{tabular}{*{6}{|c}|}
\hline
$m$&			1					&2&3				&5									&7					\\
\hline
$1/\sqrt[m]{3}$&	1/3		&0.57735&0.693361	&0.802742						&	0.854751\\
\hline
\end{tabular}
\end{table}
\noindent In Figure \ref{fig0}, we obtain the location of $r_{p,m,k}$ for certain values of $m, p\in\mathbb{N}$ and $k\in[0,1)$.
\begin{figure}[H]
\begin{minipage}[c]{1\linewidth}
\centering
\includegraphics[scale=0.9]{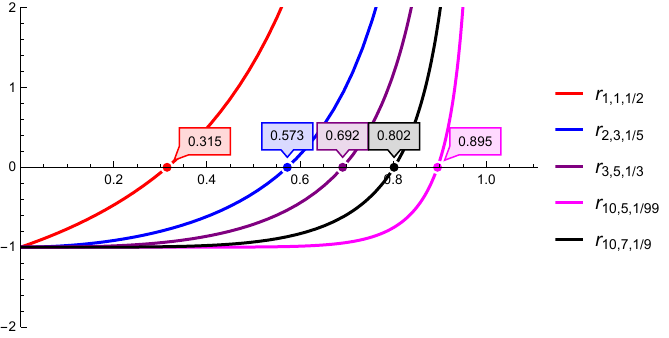}
\end{minipage}\hspace{0.5cm}
\begin{minipage}[c]{1\linewidth}
\centering
\includegraphics[scale=0.9]{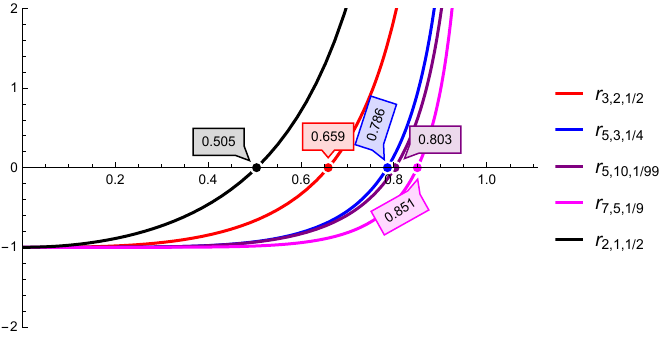}
\end{minipage}
\caption{The position of $r_{p,m,k}$, the smallest root of equation (\ref{p3}) in $(0, 1)$, varies with different values of $p, m$, and $k$}
\label{fig0}
\end{figure}
\noindent The subsequent result presents an improved version of the Bohr inequality for harmonic mappings, incorporating multiple Schwarz functions.
\begin{theo}\label{TT1} Let $f(z)=h(z)+\ol{g(z)}=\sum_{n=0}^\infty a_n z^n+\ol{\sum_{n=1}^\infty b_n z^n}$ be a sense-preserving, $K$-quasiconformal harmonic mapping in $\mathbb{D}$, and $\omega_m\in\mathcal{B}_m$ for $m\geq 1$. If $\Vert h(z)\Vert_\infty\leq 1$ in $\mathbb{D}$, then
\beas |h(\omega_m(z))|+|h'(\omega_m(z))||\omega_q(z)|+\sum_{n=2}^\infty |a_n||\omega_p(z)|^n+\sum_{n=1}^\infty |b_n||\omega_p(z)|^n\leq 1\eeas
for $r\leq r_{m,p,q,k}\leq R_{m,q}$, where $R_{m,q}\in(0,1)$ is the unique root of the equation $r^{2m}+2r^q-1=0$ and $r_{m,p,q,k}\in(0,1)$ is the smallest positive root of the equation 
\beas \frac{2r^q}{1+r^m}+\frac{2(k+ r^p)(1+r^m)r^p}{1- r^p} -(1-r^m)=0,\eeas
where $k=(K-1)/(K +1)$. The radius $r_{m,p,q,k}$ cannot be improved.
\end{theo}
\begin{proof}
As $f$ is locally univalent and $K$-quasiconformal sense-preserving harmonic mapping in $\mathbb{D}$, the classical Schwarz lemma gives that the dilatation $\omega=g'/h'$ is 
analytic in $\mathbb{D}$ and $|\omega(z)|\leq k$, {\it i.e.}, $|g'(z)|\leq k|h'(z)|$ in $\mathbb{D}$, where $k = (K-1)/(K+1)\in[0,1)$. Let $|a_0|=a\in[0,1)$.
In view of \textrm{Lemmas \ref{lem2}} and \ref{lem3}, we have $|a_n|\leq 1-a^2$ for $n\geq1$ and
\bea\label{f1} \sum_{n=1}^\infty |b_n|^2 r^n\leq k^2 \sum_{n=1}^\infty |a_n|^2 r^n\leq k^2(1-a^2)^2\frac{r}{1- r}.\eea
Using (\ref{f1}) and the Cauchy-Schwarz inequality, we have 
\bea\label{f2} \sum_{n=1}^\infty |b_n| r^n\leq \left(\sum_{n=1}^\infty |b_n|^2 r^n\right)^{1/2}\left(\sum_{n=1}^\infty r^n\right)^{1/2} \leq k(1-a^2)\frac{ r}{1- r}.\eea
As $\omega_m\in\mathcal{B}_m$, in view of the classical Schwarz Lemma, we have $|\omega_m(z)|\leq |z|^m$. Using (\ref{f2}), we have
\bea\label{f4} &&\sum_{n=1}^\infty |b_n| |\omega_p(z)|^n\leq \sum_{n=1}^\infty |b_n| r^{pn} \leq k(1-a^2)\frac{ r^p}{1- r^p}\\[2mm]\text{and}
\label{f5}&&\sum_{n=2}^\infty |a_n| |\omega_p(z)|^n\leq \sum_{n=2}^\infty |a_n| r^{pn} \leq (1-a^2)\frac{ r^{2p}}{1- r^p}.\eea
In view of \textrm{Lemmas \ref{lem1}} and \ref{lem2}, we have 
\beas |h(z)|\leq \frac{|h(0)|+|z|}{1+|h(0)||z|}= \frac{a+|z|}{1+a|z|}\quad\text{and}\quad \left|h'(z)\right|\leq \frac{1-|h(z)|^2}{1-|z|^2}.\eeas
Let $G_1(t)=(a+t)/(1+a t)$ and $G_2(t)=1/(1-t^2)$, where $0\leq t\leq \beta\; (\leq1)$ and $a\geq 0$. It is evident that $G_1'(t)=(1-a^2)/(1+at)^2\geq 0$, which shows that $G_1(t)$ is a monotonically increasing functions of $t\in[0,\beta]$. Hence, we have $G_1(t)\leq G_1(\beta)$. Similarly, we have $G_2(t)\leq G_2(\beta)$.
Thus, for $|z|=r<1$, we have 
\bea\label{f6} &&|h(\omega_m(z))|\leq \frac{a+|\omega_m(z)|}{1+a|\omega_m(z)|}\leq \frac{a+r^m}{1+ar^m}\\[2mm]\text{and}
\label{f7}&&\left|h'(\omega_m(z))\right|\leq \frac{1-|h(\omega_m(z))|^2}{1-|\omega_m(z)|^2}\leq \frac{1}{1-r^{2m}}\left(1-|h(\omega_m(z))|^2\right).\eea
Let $G_3(t)=t+\alpha(1-t^2)$, where $0\leq t\leq \beta(\leq1)$ and $\alpha\geq 0$. Then, $G_3'(t)=1-2\alpha t\geq 0$ for $\alpha\leq 1/2$. Hence, we have $G_3(t)\leq G_3(\beta)$ for $0\leq \alpha\leq 1/2$.
Therefore,
\beas &&|h(\omega_m(z))|+|h'(\omega_m(z))||\omega_q(z)|+\sum_{n=2}^\infty |a_n||\omega_p(z)|^n+\sum_{n=1}^\infty |b_n||\omega_p(z)|^n\\[2mm]
&\leq&|h(\omega_m(z))|+\frac{r^q}{1-r^{2m}}\left(1-|h(\omega_m(z))|^2\right)+(1-a^2)\frac{ r^{2p}}{1- r^p}+k(1-a^2)\frac{ r^p}{1- r^p}\eeas
\beas
&\leq& \frac{a+r^m}{1+ar^m}+\frac{r^q}{1-r^{2m}}\left(1-\left(\frac{a+r^m}{1+ar^m}\right)^2\right)+\frac{(1-a^2)(r^p+k)r^{p}}{1- r^p}\hspace{2cm}\\[2mm]
&=& \frac{a+ r^m}{1+a r^m}+\frac{(1-a^2) r^q}{(1+a r^m)^2}+\frac{(1-a^2)(r^p+k)r^{p}}{1- r^p}\\[2mm]
&=&1+\frac{(1-a) F_1(a, r)}{1+a r^m}\eeas
for $r\in[0,R_{m,q}]$, where $R_{m,q}\in(0,1)$ is the unique root of the equation $r^{2m}+2r^q-1=0$ and 
\beas F_1(a,r)=\frac{(1+a)r^q}{1+ar^m}+\frac{(1+a)(1+ar^m)(r^p+k)r^{p}}{1- r^p} -(1-r^m).\eeas
Differentiating $F_1(a, r)$ partially with respect to $a$, we obtain 
\beas\frac{\pa}{\pa a}F_1(a, r)&=&\frac{r^q(1-r^m)}{(1+a r^m)^2}+\frac{(r^p+k)r^{p}(1+r^m+2 a r^m)}{1 - r^p}\geq 0,\eeas
which shows that $F_1(a, r)$ is a monotonically increasing function of $a\in[0,1)$ and it follows that 
\beas F_1(a,r)\leq \lim_{a\to1^-}F_1(a,r)=\frac{2r^q}{1+r^m}+\frac{2(r^p+k)(1+r^m)r^{p}}{1- r^p} -(1-r^m)\leq 0\eeas
for $r\leq r_{m,p,q,k}$, where $r_{m,p,q,k}\in(0,1)$ is the smallest positive root of the equation 
\bea\label{f3}\frac{2r^q}{1+r^m}+\frac{2 (k+r^p)(1+r^m)r^p}{1- r^p} -(1-r^m)=0,\eea
where $k = (K-1)/(K +1)$.
We claim that $r_{m,p,q,k}\leq R_{m,q}$. For $r>R_{m,q}$, we have $r^{2m}+2r^q-1>0$ and 
\beas \frac{2r^q}{1+r^m}+\frac{2(r^p+k)(1+r^m)r^{p}}{1- r^p} -(1-r^m)=\frac{r^{2m}+2r^q-1}{1+r^m}+ \frac{2(r^p+k)(1+r^m)r^{p}}{1- r^p}>0.\eeas
This shows that $r_{m,p,q,k}\leq R_{m,q}$.\\[2mm]
\indent To prove the sharpness of the radius $r_{m,p,q,k}$, we construct an extremal function that attains the given bounds. Consider the function $f_1(z)=h_1(z)+\ol{g_1(z)}$ in $\mathbb{D}$ such that 
\beas h_1(z)=\frac{a+z}{1+az}=A_0+\sum_{n=1}^\infty A_n z^n\quad\text{and}\quad g_1(z)=\lambda k \sum_{n=1}^\infty A_n z^n\eeas
with $\omega_p(z)=z^p$ for $p\geq 1$, where $A_0=a$, $A_n=(1-a^2)(-a)^{n-1}$ for $n\geq 1$, $a\in[0,1)$, $|\lambda|=1$ and $k=(K-1)/(K+1)$.
For $z=r$, we have 
\beas&&|h_1(r^m)|+|h_1'(r^m)| r^q+\sum_{n=2}^\infty |A_n| r^{pn}+\sum_{n=1}^\infty |k\lambda A_n| r^{pn}\\[2mm]
&&=\frac{a+r^m}{1+a r^m}+\frac{(1-a^2) r^q}{(1+a r^m)^2}+(1-a^2) r^p \sum_{n=2}^\infty (a r^p)^{n-1}+(1-a^2)k r^p\sum_{n=1}^\infty (a r^p)^{n-1}\\[2mm]
&&=\frac{a+ r^m}{1+a r^m}+\frac{(1-a^2) r^q}{(1+a r^m)^2}+\frac{(1-a^2)a r^{2p}}{1-a r^p}+\frac{(1-a^2)k r^p}{1-a r^p}\\[2mm]
&&=1+(1-a) F_2(a, r),\eeas
where 
\beas F_2(a, r)=\frac{(1+a) r^q}{(1+a r^m)^2}+\frac{(1+a)a r^{2p}}{1-a r^p}+\frac{(1+a)k r^p}{1-a r^p}-\frac{(1- r^m)}{1+a r^m}.\eeas
It is evident that
\beas \lim_{a\to1^-} F_2(a, r)=\frac{1}{1+r^m}\left(\frac{2r^q}{1+r^m}+\frac{2(k+r^p)(1+r^m)r^p}{1-r^p}-(1- r^m)\right)>0\eeas
for $ r> r_{m,p,q,k}$, where $r_{m,p,q,k}$ is the smallest positive root of the equation (\ref{f3}) in $(0,R_{m,p})$. This shows that the number $r_{m,p,q,k}$ cannot be improved. This completes the proof.
\end{proof}
\begin{rem}
We discuss some special cases of \textrm{Theorem \ref{TT1}} as well as several helpful observations.
\begin{enumerate}
\item[(i)] Setting $p=q=1$, $m\to\infty$ and $\omega_1(z)=z$ in \textrm{Theorem \ref{TT1}} gives the first part of \textrm{Theorem B}.
\item[(ii)] Setting $m=p=q=1$ and $\omega_1(z)=z$ in \textrm{Theorem \ref{TT1}} gives \textrm{Theorem 3.2} of \cite{BM2024}.
\end{enumerate}
\end{rem}
\noindent In Table \ref{tab1} and Figure \ref{fig1}, we obtained the values of $R_{m,q}$ and $r_{m,p,q,k}$, respectively, for certain values of $m, p, q\in\mathbb{N}$ and $k\in[0,1)$. Table \ref{tab1} lists the values of $R_{m,q}$ that represent the theoretical upper bounds for the computed radius $r_{m,p,q,k}$. As shown in Figure \ref{fig1}, $r_{m,p,q,k}$ consistently falls within these bounds for various parameter settings.
\begin{table}[H]
\centering
\caption{$R_{m,q}\in(0,1)$ is the unique root of equation $r^{2m}+2r^q-1=0$}
\label{tab1}
\begin{tabular}{*{6}{|c}|}
\hline
$m$ &1&3&3&5&10\\
\hline
q&1&3&2&30&30\\
\hline
$R_{m,q}$&$\sqrt{2}-1$&0.745432&0.673348&0.948565&0.958906\\
\hline
\end{tabular}
\end{table}
\begin{figure}[H]
\centering
\includegraphics[scale=0.9]{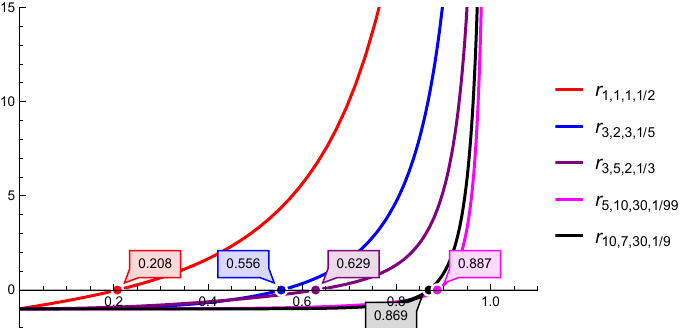}
\caption{The position of $r_{m,p,q,k}$ in $(0,1)$ for different values of $m,p,q\in\mathbb{N}$ and $k\in[0, 1)$}
\label{fig1}
\end{figure}
\noindent The subsequent result presents an improved version of the Bohr inequality for harmonic mappings, incorporating multiple Schwarz functions.
\begin{theo}\label{T7} Let $f(z)=h(z)+\ol{g(z)}=\sum_{n=0}^\infty a_n z^n+\ol{\sum_{n=1}^\infty b_n z^n}$ be a sense-preserving, $K$-quasiconformal harmonic mapping in $\mathbb{D}$, and $\omega_p\in\mathcal{B}_p$ for $p\geq 1$. If $\Vert h(z)\Vert_\infty\leq 1$ in $\mathbb{D}$, then
\beas |h(\omega_m(z))|^2+|h'(\omega_m(z))||\omega_q(z)|+\sum_{n=2}^\infty |a_n||\omega_p(z)|^n+\sum_{n=1}^\infty |b_n||\omega_p(z)|^n\leq 1\eeas
for $r\leq r_{m,p,q,k}\leq R_{2,m,p}$, where $R_{2,m,p}\in(0,1)$ is the unique root of the equation $1-r^{2m}-r^q=0$ and $r_{m,p,q,k}\in(0,1)$ is the smallest positive root of the equation 
\beas -\frac{(1-r^{2m}-r^q)}{(1+r^m)^2}+\frac{(r^p+k)r^{p}}{1- r^p}=0,\eeas
where $k=(K-1)/(K+1)$. The radius $r_{m,p,q,k}$ cannot be improved.
\end{theo}
\begin{proof}
Using similar argument as in the proof of \textrm{Theorem \ref{TT1}}, and in view of \textrm{Lemmas \ref{lem1}, \ref{lem2}} and \ref{lem3}, we obtain the inequalities (\ref{f4}), (\ref{f5}), (\ref{f6}) and (\ref{f7}).
Thus, we have
\beas&& |h(\omega_m(z))|^2+|h'(\omega_m(z))||\omega_q(z)|+\sum_{n=2}^\infty |a_n||\omega_p(z)|^n+\sum_{n=1}^\infty |b_n||\omega_p(z)|^n\\[1mm]
&\leq& |h(\omega_m(z))|^2+\frac{r^q}{1-r^{2m}}\left(1-|h(\omega_m(z))|^2\right)+(1-a^2)\frac{ r^{2p}}{1- r^p}+k(1-a^2)\frac{ r^p}{1- r^p}\\[1mm]
&=&\left(1-\frac{ r^q}{1- r^{2m}}\right)|h(\omega_m(z))|^2+\frac{ r^q}{1- r^{2m}}+\frac{(1-a^2)(r^p+k)r^{p}}{1- r^p}\\[1mm]
&\leq& \frac{1- r^{2m}- r^q}{1- r^{2m}}\left(\frac{a+ r^m}{1+a r^m}\right)^2+\frac{ r^q}{1- r^{2m}}+\frac{(1-a^2)(r^p+k)r^{p}}{1- r^p}\\[1mm]
&=&1+\frac{1-r^{2m}-r^q}{1-r^{2m}}\left( \left(\frac{a+r^m}{1+ar^m}\right)^2-1\right)+\frac{(1-a^2)(r^p+k)r^{p}}{1- r^p}\\[1mm]
&=&1-\frac{(1-a^2)(1-r^{2m}-r^q)}{(1+ar^m)^2}+\frac{(1-a^2)(r^p+k)r^{p}}{1- r^p}\\[1mm]
&=&1+(1-a^2)F_3(a, r),\eeas
where the second inequality holds for such $r\in[0,1]$ satisfying $1- r^{2m}- r^q\geq 0$, {\it i.e.,} for $r\in[0, R_{2,m,q}]$, where $R_{2,m,q}\in(0,1)$ is the unique root of the equation $1-r^{2m}-r^q=0$ and 
\beas F_3(a,r)=-\frac{(1-r^{2m}-r^q)}{(1+ar^m)^2}+\frac{(r^p+k)r^{p}}{1- r^p}\eeas
with $k=(K-1)/(K+1)$.
It is evident that $F_3(a, r)$ is a monotonically increasing function of $a\in[0,1)$ and it follows that 
\beas F_3(a, r)\leq \lim_{a\to1^-}F_3(a, r)=-\frac{(1-r^{2m}-r^q)}{(1+r^m)^2}+\frac{(r^p+k)r^{p}}{1- r^p}\leq 0\eeas
for $ r\leq r_{m,p,q,k}$ and $r_{m,p,q,k}$ is the smallest positive root of the equation 
\bea\label{f8}F_{m,p,q,k}( r):=-\frac{(1-r^{2m}-r^q)}{(1+r^m)^2}+\frac{(r^p+k)r^{p}}{1- r^p}=0.\eea
It is evident that $1- r^{2m}- r^q<0$ for $r>R_{2,m,q}$ and hence, we have $F_{m,p,q,k}( r)>0$ for $r>R_{2,m,q}$. Therefore, we have $r_{m,p,q,k}\leq R_{2,m,q}$. \\[2mm]
\indent 
To prove the sharpness of the radius $r_{m,p,q,k}$, we construct an extremal function that attains the given bounds. Consider the function $f_2(z)=h_2(z)+\ol{g_2(z)}$ in $\mathbb{D}$ such that 
\beas h_2(z)=\frac{a+z}{1+az}=A_0+\sum_{n=1}^\infty A_n z^n\quad\text{and}\quad g_2(z)=\lambda k \sum_{n=1}^\infty A_n z^n\eeas
with $\omega_p(z)=z^p$ for $p\geq 1$, where $A_0=a$, $A_n=(1-a^2)(-a)^{n-1}$ for $n\geq 1$, $a\in[0,1)$, $|\lambda|=1$ and $k=(K-1)/(K+1)$.
For $z=r$, we have 
\beas&&|h_2(r^m)|^2+|h_2'(r^m)| r^q+\sum_{n=2}^\infty |A_n| r^{pn}+\sum_{n=1}^\infty |k\lambda A_n| r^{pn}\hspace{2cm}\\[1mm]
&=&\left(\frac{a+r^m}{1+a r^m}\right)^2+\frac{(1-a^2) r^q}{(1+a r^m)^2}+(1-a^2) r^p \sum_{n=2}^\infty (a r^p)^{n-1}+(1-a^2)k r^p\sum_{n=1}^\infty (a r^p)^{n-1}\\[2mm]
&=&1-\frac{(1-a^2)(1-r^{2m})}{(1+a r^m)^2}+\frac{(1-a^2) r^q}{(1+a r^m)^2}+\frac{(1-a^2) a r^{2p}}{1-a r^p}+\frac{(1-a^2)k r^p}{1-a r^p}\\[2mm]
&=&1+(1-a^2)F_4(a, r),\eeas
where 
\beas F_4(a, r)=-\frac{1-r^{2m}-r^q}{(1+a r^m)^2}+\frac{(k+a r^p)r^p}{1-a r^p}.\eeas
It is evident that
\beas \lim_{a\to1^-} F_4(a, r)=-\frac{1-r^{2m}-r^q}{(1+r^m)^2}+\frac{(k+r^p)r^p}{1-r^p}>0\quad\text{for}\quad r> r_{m,p,q}, \eeas
for $r>r_{m,p,q,k}$, where $r_{m,p,q,k}$ is the smallest positive root of the equation (\ref{f8}) in $[0, R_{2,m,q}]$. This shows that the number $r_{m,p,q,k}$ cannot be improved. This completes the proof.
\end{proof}
\begin{rem}\label{rem4}
We discuss some special cases of \textrm{Theorem \ref{T7}} as well as several helpful observations.
\begin{enumerate}
\item[(i)] Setting $p=q=1$, $m\to\infty$ and $\omega_1(z)=z$ in \textrm{Theorem \ref{T7}} gives the second part of \textrm{Theorem B}.
\item[(ii)] By setting $m=p=q=1$ and $\omega_1(z)=z$ in \textrm{Theorem \ref{T7}}, we obtain \textrm{Theorem 3.3} of \cite{BM2024}.
\end{enumerate}
\end{rem}
\noindent In Table \ref{tab2} and Figure \ref{fig2}, we obtain the values of $R_{2,m,q}$ and $r_{m,p,q,k}$, respectively, for certain values of $m,p,q\in\mathbb{N}$ and $q\in[0,1)$. From Table \ref{tab2} and Figure \ref{fig2}, it can be observed that the inequality $r_{m,p,q,k}\leq R_{2,m,q}$ is satisfied for arbitrary choices of $m,p,q,k$.
\begin{table}[H]
\centering
\caption{$R_{2,m,q}\in(0,1)$ is the unique root of the equation $1-r^{2m}-r^q=0$}
\label{tab2}
\begin{tabular}{*{6}{|c}|}
\hline
$m$&1&3&3&5&10\\
\hline
$q$&1&3&2&30&30\\
\hline
$R_{2,m,q}$&$\left(\sqrt{5}-1\right)/2$&0.8518&0.826031&0.962497&0.972272\\
\hline
\end{tabular}
\end{table}
\begin{figure}[H]
\centering
\includegraphics[scale=0.8]{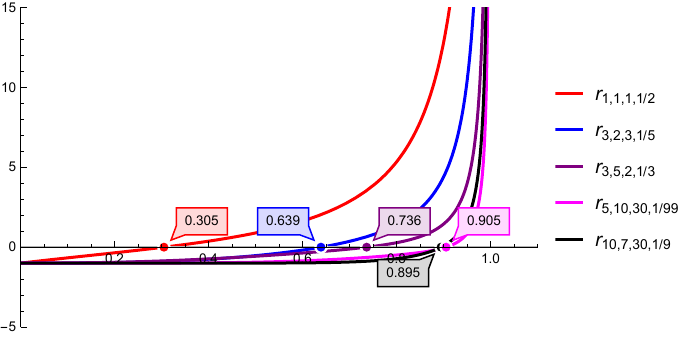}
\caption{The position of $r_{m,p,q,k}$ in $(0,1)$ for different values of $m, p,q\in\mathbb{N}$ and $k\in[0,1)$}
\label{fig2}
\end{figure}
\noindent In the subsequent result, we obtain a refined version of the Bohr inequality for harmonic mappings, incorporating multiple Schwarz functions.
\begin{theo}\label{T1} Let $f(z)=h(z)+\ol{g(z)}=\sum_{n=0}^\infty a_n z^n+\ol{\sum_{n=1}^\infty b_n z^n}$ be a sense-preserving, $K$-quasiconformal harmonic mapping in $\mathbb{D}$, and $\omega_k\in\mathcal{B}_k$ for $k\geq 1$. If $h(z)$ is bounded in $\mathbb{D}$, then
\beas &&\sum_{n=0}^\infty |a_n||\omega_p(z)|^n+\sum_{n=1}^\infty |b_n| |\omega_p(z)|^n+\left(\frac{1}{1+|a_0|}+\frac{|\omega_p(z)|}{1-|\omega_p(z)|}\right)\sum_{n=1}^\infty |a_n|^2 |\omega_p(z)|^{2n}\\
&&+\frac{(1-r_p^{2m})^2}{8r_p^{2m}}\sum_{n=1}^\infty n|a_n|^2 |\omega_m(z)|^{2n}\leq \Vert h(z)\Vert_\infty\eeas 
for $r\leq r_{p,k}$, where $r_{p,k}\in(0,1)$ is the unique root of the equation $(2k+3)r^p-1=0$ with $k=(K-1)/(K+1)$.
The numbers $r_{p,k}$ and $(1-r_p^{2m})^2/(8r_p^{2m})$ cannot be improved.
\end{theo}
\begin{proof} For simplicity, let $\Vert h(z)\Vert_\infty\leq 1$.
Using similar argument as in the proof of \textrm{Theorem \ref{TT1}} and \textrm{Lemma \ref{lem3}}, we obtain the inequality (\ref{f4}).
In view of \textrm{Lemma \ref{lem5}}, we have
\bea\label{f9} \sum_{n=1}^\infty |a_n| r^n+\left(\frac{1}{1+a}+\frac{r}{1- r}\right)\sum_{n=1}^\infty |a_n|^2  r^{2n}\leq \frac{(1-a^2)r}{1- r}.\eea
Let $G_4(t)=t/(1-t)$, where $0\leq t\leq \beta(\leq1)$ and $a\geq 0$. It is evident that $G_4(t)$ is a monotonically increasing function of $t\in[0,\beta]$, and it follows that $G_4(t)\leq G_4(\beta)$.
Using (\ref{f9}), we have
\bea\label{f10} &&\sum_{n=1}^\infty |a_n||\omega_p(z)|^n+\left(\frac{1}{1+|a_0|}+\frac{|\omega_p(z)|}{1-|\omega_p(z)|}\right)\sum_{n=1}^\infty |a_n|^2 |\omega_p(z)|^{2n}\nonumber\\[2mm]
&\leq&\sum_{n=1}^\infty |a_n|r^{np}+\left(\frac{1}{1+a}+\frac{r^p}{1-r^p}\right)\sum_{n=1}^\infty |a_n|^2 r^{2pn}\nonumber\\[2mm]
&\leq&\frac{(1-a^2)r^p}{1- r^p}.\eea
Using (\ref{f4}) and (\ref{f10}), we have 
\beas &&\sum_{n=0}^\infty |a_n||\omega_p(z)|^n+\sum_{n=1}^\infty |b_n| |\omega_p(z)|^n+\left(\frac{1}{1+|a_0|}+\frac{|\omega_p(z)|}{1-|\omega_p(z)|}\right)\sum_{n=1}^\infty |a_n|^2 |\omega_p(z)|^{2n}\\
&&+\mu\sum_{n=1}^\infty n|a_n|^2 |\omega_m(z)|^{2n}\\
&\leq&a+\frac{(1-a^2)r^p}{1- r^p}+k(1-a^2)\frac{ r^p}{1- r^p}+\mu(1-a^2)^2\sum_{n=1}^\infty nr^{2mn}\\[2mm]
&=&a+\frac{(1-a^2)r^p}{1- r^p}+k(1-a^2)\frac{ r^p}{1- r^p}+\mu(1-a^2)^2\frac{r^{2m}}{(1-r^{2m})^2}\\
&=&1+F_5(a, r)\eeas
where 
\beas F_5(a,r)=\frac{(1-a^2)}{2}\left(1+\left(\frac{2(k+1)r^p}{1-r^p}-1\right)+2\mu (1-a^2)\frac{r^{2m}}{(1-r^{2m})^2}-\frac{2}{1+a}\right),\eeas
where $k=(K-1)/(K+1)$.
Differentiating partially $F_5(a,r)$ with respect to $r$, we obtain
\beas \frac{\pa}{\pa r}F_5(a,r)=\frac{(1-a^2)(k+1)pr^{p-1}}{(1- r^p)^2}+\lambda(1-a^2)^2\frac{2 m r^{2m-1}(1+r^{2m})}{(1-r^{2m})^3}\geq 0,\eeas
which shows that $F_5(a,r)$ is a monotonically increasing function of $r\in[0, 1)$ and hence, we have $F_5(a,r)\leq F_5(a,r_{p,k})$ for $r\leq r_{p,k}$, where $r_{p,k}\in(0, 1)$ is the unique positive root of the equation $(2k+3)r^p-1=0$.
Evidently, $F_5(a,r_{p,k})=(1-a^2)F_6(a)/2$, where
\beas F_6(a)=1+2\mu (1-a^2)\frac{r_{p,k}^{2m}}{\left(1-r_{p,k}^{2m}\right)^2}-\frac{2}{1+a}.\eeas
Note that 
\beas F_6(0)=2\mu\frac{r_{p,k}^{2m}}{\left(1-r_{p,k}^{2m}\right)^2}-1\quad\text{and}\quad \lim_{a\to1^-}F_6(a)=0.\eeas
Differentiating $F_6(a)$ with respect to $a$, we obtain
\beas F_6'(a)=\frac{2}{(1+a)^2}\left(1-2\mu a(1+a)^2\frac{r_{p,k}^{2m}}{\left(1-r_{p,k}^{2m}\right)^2}\right)
\geq \frac{2}{(1+a)^2}\left(1-\frac{8\mu r_{p,k}^{2m}}{\left(1-r_{p,k}^{2m}\right)^2}\right)\geq 0\eeas
for $\mu\leq (1-r_{p,k}^{2m})^2/(8r_{p,k}^{2m})$. Therefore, $F_6(a)$ is a monotonically increasing function of $a\in[0,1)$, whenever $\mu\leq (1-r_{p,k}^{2m})^2/(8r_{p,k}^{2m})$. Hence, we 
have $F_6(a)\leq \lim_{a\to1^-}F_6(a)=0$ for $a\in[0,1)$ and $\mu\leq (1-r_{p,k}^{2m})^2/(8r_{p,k}^{2m})$. Therefore, we have
\beas&& \sum_{n=0}^\infty |a_n||\omega_p(z)|^n+\sum_{n=1}^\infty |b_n| |\omega_p(z)|^n+\left(\frac{1}{1+|a_0|}+\frac{|\omega_p(z)|}{1-|\omega_p(z)|}\right)\sum_{n=1}^\infty |a_n|^2 |\omega_p(z)|^{2n}\\
&&+\frac{(1-r_{p,k}^{2m})^2}{8r_{p,k}^{2m}}\sum_{n=1}^\infty n|a_n|^2 |\omega_m(z)|^{2n}\leq 1\eeas
for $r\leq r_{p,k}$, where $r_{p,k}\in(0, 1)$ is the unique root of the equation $(2k+3)r^p-1=0$.\\[2mm]
\indent To prove the sharpness of the radius $r_{p,k}$, we construct an extremal function that attains the given bounds.
Consider the function $f_3(z)=h_3(z)+\ol{g_3(z)}$ in $\mathbb{D}$ such that 
\beas h_3(z)=\frac{a+z}{1+az}=A_0+\sum_{n=1}^\infty A_n z^n\quad\text{and}\quad g_3(z)=\lambda k \sum_{n=1}^\infty A_n z^n\eeas
with $\omega_p(z)=z^p$ for $p\geq 1$, where $A_0=a\in[0,1)$, $A_n=(1-a^2)(-a)^{n-1}$ for $n\geq 1$, $|\lambda|=1$ and $k=(K-1)/(K+1)$.
Thus,
\beas && \sum_{n=0}^\infty |A_n||r^p|^n+\sum_{n=1}^\infty |k\lambda A_n| |r^p|^n+\left(\frac{1}{1+|A_0|}+\frac{|r^p|}{1-|r^p|}\right)\sum_{n=1}^\infty |A_n|^2 |r^p|^{2n}\\[1mm]
&&+\frac{\left(1-r_{p,k}^{2m}\right)^2}{8r_{p,k}^{2m}}\sum_{n=1}^\infty n|A_n|^2 |r^m|^{2n}\\[1mm]
&=&a+\frac{(1+k)(1-a^2)}{a}\sum_{n=1}^\infty (a r^p)^{n}+\frac{(1-a^2)^2(1+a r^p) r^{2p}}{(1+a)(1- r^p)}\sum_{n=1}^\infty (a r^p)^{2(n-1)}\\[1mm]
&&+\frac{\left(1-r_{p,k}^{2m}\right)^2}{8r_{p,k}^{2m}}(1-a^2)^2r^{2m}\sum_{n=1}^\infty n (ar^{m})^{2(n-1)} \\[1mm]
&=&a+\frac{(1+k)(1-a^2) r^p}{1-ar^p}+\frac{(1-a^2)(1-a)(1+a r^p) r^{2p}}{(1- r^p)(1-a^2 r^{2p})}+\frac{\left((1-a^2)(1-r_{p,k}^{2m}) r^{m}\right)^2}{8r_{p,k}^{2m}(1-a^2r^{2m})^2} \\[1mm]
&=& 1+(1-a) F_7(a, r),\eeas
where 
\beas F_7(a, r)=\frac{(1+k)(1+a) r^p}{1-ar^p}+\frac{(1-a^2)r^{2p}}{(1- r^p)(1-a r^{p})}+\frac{(1-a^2)(1+a)(1-r_p^{2m})^2 r^{2m}}{8r_p^{2m}(1-a^2r^{2m})^2}-1.\eeas
Differentiating $F_7(a, r)$ partially with respect to $r$, we obtain
\beas\frac{\pa}{\pa r} F_7(a, r)&=&\frac{(1+a)(1+k) p\; r^{p-1}}{(1-a r^p)^2}+(1-a^2)\frac{p r^{2 p-1} \left(2-(a+1) r^p\right)}{\left(1-r^p\right)^2 \left(1-a r^p\right)^2}\\[2mm]
&&+\frac{2 m (1-a^2)(1+a)(1-r_p^{2m})^2\left(1+a^2 r^{2 m}\right) r^{2 m-1} }{8r_p^{2m}\left(1-a^2 r^{2 m}\right)^3}>0\eeas
for $r\in(0,1)$. Therefore, $F_7(a, r)$ is a strictly increasing function of $r\in(0,1)$. Therefore, for $r>r_p$, we have $F_7(a, r)>F_7(a, r_p)$. It is evident that
\beas \lim_{a\to1^-}F_7(a,r_p)=\frac{2(1+k)r_p^{p}}{1-r_p^p}-1=0.\eeas
This shows that the numbers $r_p$ and $(1-r_p^{2m})^2/(8r_p^{2m})$ are the best possible. This completes the proof.
\end{proof}
\begin{rem}
Setting $p=m$ in \textrm{Theorem \ref{T1}} gives $(1-r_p^{2m})^2/(8r_p^{2m})=\frac{8K^2 \left(3 K+1\right)^2}{(K+1)^2 (5 K+1)^2}$. 
\end{rem}
\section{Discussion and Concluding Remarks}
\noindent In this study, we affirmatively answered Questions \ref{Q1} and \ref{Q2} by establishing a series of results that form a unified theory for Bohr's inequality in the context of  
$K$-quasiconformal harmonic mappings. An important accomplishment is the derivation of a refined version of Theorem C (via Theorem \ref{T6} with $p=m=1$), thereby resolving 
the previously unsolved problem. Furthermore, Theorems \ref{TT1} and \ref{T7} present improved versions by incorporating both the function's value and its derivative, whereas 
Theorem \ref{T1} provides a refined version by including area-type terms that push the boundaries of classical inequality. As demonstrated in Remarks \ref{rem2}-\ref{rem4}, our 
framework systematically retrieves the foundational results of Theorems A, B, and C by selecting appropriate parameters. The data presented in Tables \ref{tab0}-\ref{tab2} and 
Figures \ref{fig0}-\ref{fig2} demonstrate a consistent alignment with our theoretical conclusions, thereby providing a visual representation of the behavior of our calculated radii 
across a range of parameter settings. This work opens up several possibilities for future research:\\[1mm]
(i) Is it possible to achieve similar results for classes of harmonic mappings associated with starlike or convex functions \cite{AH2021, 1AH2021}?\\[1mm]
(ii) Is it possible to establish a Bohr phenomenon for pluriharmonic mappings in several complex variables using a similar multi-Schwarz function approach?\\[1mm]
(iii) Can the techniques be extended to incorporate the Bohr-Rogosinski radius (e.g. \cite{AKP2020, KKP2021}) for harmonic mappings with multiple Schwarz functions?
\section*{Declarations}
\noindent{\bf Acknowledgment:} The work of the first author is supported by University Grants Commission (IN) fellowship (No. F. 44 - 1/2018 (SA - III)).\\[2mm]
{\bf Conflict of Interest:} The authors declare no conflicts of interest regarding the publication of this paper.\\[1mm]
{\bf Availability of data and materials:} Not applicable.

\end{document}